\documentclass[reqno]{amsart}
\usepackage[scale=0.75, centering, headheight=14pt]{geometry}
\usepackage[latin1]{inputenc}
\usepackage[T1]{fontenc}
\usepackage{lmodern}
\usepackage[english]{babel}
\usepackage{stmaryrd}
\usepackage{tikz}
\usepackage{bbm}
\usepackage{amsmath,amssymb,amsfonts,amsthm}
\usepackage{mathtools,accents}
\usepackage{mathrsfs}
\usepackage{xfrac}
\usepackage{array} 
\usepackage{aliascnt}
\usepackage{booktabs} 
\usepackage{array} 

\usepackage{verbatim} 
\usepackage{subfig} 
\usepackage{mathrsfs}
\usepackage{mathrsfs, dsfont}
\usepackage{amssymb}
\usepackage{amsthm}
\usepackage{amsmath,amsfonts,amssymb,esint}
\usepackage{graphics,color}
\usepackage{enumerate}
\usepackage{mathtools,centernot}

\usepackage{microtype}
\usepackage{paralist} 
\usepackage{cases}
\usepackage[initials]{amsrefs}
\allowdisplaybreaks

\usepackage{braket}
\usepackage{bm}

\usepackage[citecolor=blue,colorlinks]{hyperref}
\addto\extrasenglish{}

\usepackage{enumerate}
\usepackage{xcolor}

\usepackage{aliascnt}

\makeatletter
\def\newaliasedtheorem#1[#2]#3{
  \newaliascnt{#1@alt}{#2}
  \newtheorem{#1}[#1@alt]{#3}
  \expandafter\newcommand\csname #1@altname\endcsname{#3}
}
\makeatother

\theoremstyle{plain}
\newtheorem{theorem}{Theorem}[section]
\newaliasedtheorem{lemma}[theorem]{Lemma}
\newaliasedtheorem{prop}[theorem]{Proposition}
\newaliasedtheorem{claim}[theorem]{Claim}
\newaliasedtheorem{corollary}[theorem]{Corollary}

\theoremstyle{remark}

\theoremstyle{definition}
\newaliasedtheorem{definition}[theorem]{Definition}
\newaliasedtheorem{example}[theorem]{Example}

\theoremstyle{remark}
\newaliasedtheorem{remark}[theorem]{Remark}
\newaliasedtheorem{ex}[theorem]{Example}

\numberwithin{equation}{section}

\def\eps{\varepsilon}

\def\R{\mathbb R}
\def\N{{\mathbb N}}

\DeclareMathOperator*{\esssup}{ess\,sup}
\DeclareMathOperator{\dv}{div}
\DeclareMathOperator{\Lip}{Lip}

\DeclareMathOperator{\loc}{loc}

\title{On the 
commutativity of flows of rough vector fields}

\author[M. Colombo and  R. Tione]{Maria Colombo \and Riccardo Tione}
\address{Maria Colombo
\hfill\break EPFL B, Station 8, CH-1015 Lausanne, CH}
\email{maria.colombo@epfl.ch}
\address{Riccardo Tione  
\hfill\break  EPFL B, Station 8, CH-1015 Lausanne, CH}
\email{riccardo.tione@epfl.ch}

\begin{document}

\maketitle

\begin{abstract}
In the class of Sobolev vector fields in $\R^n$ of bounded divergence, for which the theory of DiPerna and Lions provides a well defined notion of flow, we characterize the vector fields whose flow commute in terms of the Lie bracket and of a regularity condition on the flows themselves. This extends a classical result of Frobenius in the smooth setting.
\end{abstract}

\section{Introduction}

Let $X,Y \in C^\infty\cap L^\infty(\R^n,\R^n)$ and let 
$\Phi_t^X$ and $\Phi_s^Y$ are the flows of the two fields, namely the solutions of the ODE
\begin{equation}
\begin{cases}
\partial_t \Phi_t^X(x) = X(\Phi_t^X(x))
\\
\Phi_0^X(x) = x
\end{cases}
\end{equation}
with vector fields $X$ and $Y$, respectively. Then, a classical result in Dynamical Systems states that:
\begin{equation}\label{result}
\Phi_t^X\circ\Phi_s^Y = \Phi_s^Y\circ \Phi_t^X, \forall t,s \in \R \qquad \Longleftrightarrow \qquad [X,Y] \equiv 0,
\end{equation}
where the bracket $[X,Y]$ is defined as
\[
[X,Y]\doteq DYX-DXY.
\]
The equivalence \eqref{result} has already been settled for locally Lipschitz fields in \cite{COMM}.
The aim of this paper is to study the validity of \eqref{result} in the weak setting of Regular Lagrangian Flows. This theory was mainly developed thanks to the effort of DiPerna \& Lions and Ambrosio, in the seminal papers \cite{DPL,AMI}. 

\begin{theorem}\label{thm:main}
Let $p \ge 1$, $X,Y \in W^{1,p}_{\rm loc}\cap L^\infty(\R^n,\R^n)$ be vector fields with bounded divergence and let $\Phi_t^X,\Phi_s^Y$ be their Regular Lagrangian Flows. The following two assertions are equivalent:
\begin{enumerate}
\item the flows of $X$ and $Y$ commute
$$\Phi_t^X\circ\Phi_s^Y = \Phi_s^Y\circ \Phi_t^X \qquad \mbox{for every } s,t\in \R,\;\mbox{for a.e. } z\in \R^n;$$
\item the flow map $\Phi_t^X$ is weakly (Lie) differentiable in direction $Y$ with locally bounded derivative (see Definition~\ref{defn:weakdiff} below) for all $t\in \R$ and the Lie bracket $[X,Y]$ is $0$.
\end{enumerate}
%
\end{theorem}

The question of commutativity of the flow arises naturally in various fields even in the nonsmooth setting. For instance, it is linked to the Frobenius theorem on involutive distributions of planes, see \cite{AM17,AM20} and the references therein, and it also appears in the context of analysis on nonsmooth RCD spaces, as in \cite{GRCD}.
\\
\\
The assumption on the boundedness of $X$ and $Y$ is helpful to avoid technical issues and is quite common in the literature, see for instance \cite{DLC}, but it is not essential for the proof. The statement of Theorem~\ref{thm:main} differs from the statement for smooth (or Lipschitz) vector fields because in the second of the two equivalent conditions, a certain regularity of the first flow appears.  It is clear that, if one of the two vector fields of Theorem~\ref{thm:main} is assumed to be Lipschitz, for instance $X$, then $\Phi_t^X$ is Lipschitz as well, so in particular $D\Phi_t^XY \in L^\infty$. Under this assumption we recover \eqref{result}, proving a more general version of \cite{COMM}, since only one of the two vector fields is Lipschitz. Notice moreover that in this case we also find the weak differentiability of the second vector field with respect to the first, i.e. $D\Phi_t^YX \in L^\infty$. 
\\
\\
For divergence-free Sobolev vector fields several constructions \cite{ACMM,JAB,Mar20} have shown that the flow may be non-differentiable, namely may not belong to any Sobolev space; {in particular}, \cite{Mar20} considers an autonomous vector field (as in our situation) in the lowest possible dimension $n=2$. However, the key feature of this theorem is that the commutativity of flows carries a nontrivial regularity information on the differentiability of the flow, which in turn allows to establish a full equivalence.
\\
\\
Except for the Lipschitz case, the only situation of which we are aware where $W^{1,p}$ differentiability of the flow can be proved is in dimension $2$, under some additional assumptions on $X$, such as $X$ continuous and always non-zero. This will be shown in the forthcoming paper of E. Marconi, \cite{Mar20}. Also in this case, Theorem~\ref{thm:main} implies then the validity of \eqref{result}, as we will show in Section \ref{EL}.
\\
\\
It was shown in a series of works \cite{LL,ALM,DLC,JAB1} that flows of Sobolev vector fields with bounded divergence have a notion of differentiability in a measure theoretic sense (see also \cite{BDN} for the case of BV vector fields).
We underline in the next corollary that, in the case of two vector fields $X$ and $Y$ as in Theorem \ref{thm:main}, the weak differentiability established in (ii) relates to the measure theoretic differential.

\begin{corollary}\label{chainrule}
Let $p, X, Y, \Phi_t^X,\Phi_s^Y$ be as in  Theorem~\ref{thm:main}. Let moreover $q \in [1,+\infty)$. Suppose the flows of $X$ and $Y$ commute. Then,
\[
\Delta_h(t,s,z)\doteq \frac{\Phi_t^X(\Phi_{s+h}^Y(z))-\Phi_t^X(\Phi_s^Y(z)))}{h}
\]
converges {weakly} in $L^q_{\loc}$ to the weak (Lie) derivative {$f = Y\circ \Phi_t^X$} of $\Phi_t^X$ in direction $Y$. {If $p > 1$, then the convergence is in the strong (local) $L^q$ topology and, denoting by $d\Phi_t^X(z)$ the approximate differential of $z\mapsto\Phi_t^X$ at $z$, the following holds almost everywhere: 
\begin{equation}\label{id}
f(z)=d\Phi_t^X(z)Y(z) = Y(\Phi_t^X(z)).
\end{equation}}
\end{corollary}

{Before explaining the strategy of the proof of the main Theorem, let us make some comments about the previous corollary. First of all, the difference between the case $p > 1$ and $p = 1$ stems from the fact that in the former one knows the approximate differentiability of the flow, while in the latter this is, to the best of our knowledge, not known. In fact, a result of \cite{JAB1} proves the \emph{differentiability in measure} of the flow, a strictly weaker notion than approximate differentiability, see \cite{AM}. Moreover, as observed for instance in \cite{BDN}, the approximate differential $d\Phi_t^X(z)$ solves the \emph{classical} ODE
\[
\partial_t (d\Phi_t^X) = DX\circ \Phi_t^X(d\Phi_t^X)
\]
for a.e. $z$. Therefore one can see that the equality $d\Phi_t^X(z)Y(z) = Y(\Phi_t^X(z))$ stated in \eqref{id} holds under the weaker assumption that $[X,Y]=0$, since these functions solve the same ODE (see \eqref{pdemixed} below) in time for a.e. $z$.
Furthermore, we observe that that previous corollary can be applied with $X=Y$ since the assumption of the commutativity of flows simply reduces to the group property of the flow $\Phi_s^Y$ and is therefore satisfied.} Hence, the flow of $Y$ is weakly differentiable in direction $Y$and its derivative is given by $Y\circ \Phi_s^Y$,
namely
\begin{equation}\label{YY}
-\int_{\R^n}(\Phi_s^Y,D\varphi Y)dz = \int_{\R^n}(Y(\Phi_s^Y(z)),\varphi(z))dz + \dv(Y)(\varphi,\Phi_s^Y)dz,\quad\forall s \in \R, \forall\varphi \in C^\infty_c(\R^n,\R^n).
\end{equation}
Moreover, if $p > 1$, the weak derivative can be rewritten in terms of the measure-theoretic differential of the flow as $d\Phi_s^Y(z)Y(z)$ and coincides with the strong $L^1$ limit of the incremental quotients
$\frac{\Phi_{h}^Y(z)-z}{h}$ as $h \to 0$.
\newline


%
%


The strategy to prove Theorem \ref{thm:main} is the following. In the smooth case, the commutativity of the flows of $X,Y$ is equivalent to the request that
\[
\frac{d}{dt}(\Phi_t^X)_*Y = 0, \forall t \in \R,
\] 
where $(\Phi_t^X)_*Y$ denotes the pushforward of $Y$ through the flow of $X$. In other words, $Y$ must be invariant under the flow of $X$. More explicitely, the previous equations amounts to say that
\begin{equation}\label{invariance}
D\Phi_t^X(z)Y(z) = Y(\Phi_t^X(z)), \quad \forall z \in \R^n, t \in \R.
\end{equation}
The latter allow us to we introduce a family of distributions of order 1, $T_t[X,Y]$, depending on the time parameter $t$. As an intermediate step, we also introduce another family of distributions, $T_{t,s}[X,Y]$, depending on the time parameters $t,s \in \R$. These distributions serve as a bridge between the two ends of the double implication of \eqref{result}, in the sense that they have the property that
\[
[X,Y] \equiv 0 \Longleftrightarrow T_t[X,Y] = 0,\forall t \Longleftrightarrow T_{t,s}[X,Y] = 0,\forall t,s \Longleftrightarrow \Phi_t^X\circ\Phi_s^Y = \Phi_s^Y\circ \Phi_t^X, \forall t,s,
\]
in the smooth setting. In the weak setting of Regular Lagrangian Flows, we will prove in Section \ref{mod} the following implications:
\[
[X,Y] \equiv 0 \Longleftarrow T_t[X,Y] = 0,\forall t \Longleftrightarrow T_{t,s}[X,Y] = 0,\forall t,s \Longleftrightarrow \Phi_t^X\circ\Phi_s^Y = \Phi_s^Y\circ \Phi_t^X, \forall t,s.
\]
Adding the hypothesis of differentiability of $\Phi_t^X$ in the direction of $Y$ for every $t \in \R$, we will also show that
\[
[X,Y] \equiv 0 \Rightarrow T_t[X,Y] = 0, \forall t \in \R,
\]
that will conclude the proof of Theorem \ref{thm:main}.
A different approach to the validity of \eqref{result} for nonsmooth vector fields may be based on the stability of Regular Lagrangian Flows with respect to $L^1$ convergence of vector fields, through an approximation procedure. This strategy appears often in the literature, for instance to show  the group property of the flow.
Since the Lie bracket is bilinear, it is not clear how to approximate the vector fields without disrupting the constraint $[X,Y] = 0$. 
Therefore, one might try to obtain an apriori estimate of $\|\Phi^X_t\circ\Phi^Y_s - \Phi_s^Y\circ \Phi_t^X\|_r $ in terms of an integral norm of $[X,Y]$. 
Our attempts in finding such an estimate failed, and this is due to the fact that even for smooth vector fields the formula
\begin{equation}\label{commgen}
\Phi^X_t\circ\Phi^Y_s (z)- \Phi_s^Y\circ \Phi_t^X(z) = st [X,Y] (z) + o(s^2+t^2)
\quad \forall t \in \R,\forall z \in \R^n
\end{equation}
involves implicitly in the second term of the right hand side the differential of the flows of $X, Y$, which for general Sobolev vector fields are not integrable. On the contrary, such formula can be exploited for Lipschitz vector fields, as done in \cite[Lemma 4.4 and 4.5]{COMM}. 

\section{Definitions and preliminaries}
Throughout the rest of the paper, let $p, X, Y, \Phi_t^X,\Phi_s^Y$ be as in  Theorem~\ref{thm:main}.

\begin{definition}\label{defn:weakdiff}
Let $F \in L^1(\R^n; \R^d)$ and let $Y\in L^\infty(\R^n,\R^n)$ be a vector field with bounded divergence.
We say that $F$ is weakly (Lie) differentiable in the direction of $Y$ if there exists $f \in L^1(\R^{n},\R^d)$ such that for every $\varphi \in C^1_c(\R^n,\R^d)$
\begin{equation}\label{dist0}
\int_{\R^n}(F,D\varphi(z) Y) + \dv(Y) (F,\varphi(z))dz = -\int_{\R^n} (f(z),\varphi(z))dz.
\end{equation}
We call $f$ the weak (Lie) derivative of $F$ in the direction of $Y$. Moreover, we say that $F$ is weakly (Lie) differentiable in the direction of $Y$  with bounded derivative if, in addition to \eqref{dist0}, $f$ is essentially bounded.
\end{definition}

\subsection{Some facts from the DiPerna-Lions-Ambrosio theory}

In this section, we recall some well-known facts about the theory of Regular Lagrangian Flows, in the axiomatization given by Ambrosio in \cite{AMI}. We say that $\Phi^b_t$ is a Regular Lagrangian Flow associated to $b \in L^\infty(\R^n)$ if
\begin{enumerate}
\item\label{meas} For $\mathcal{L}^1$-a.e. $t \in \R$, we have $|\{z:\Phi_t^b(z) \in A\}| = 0$ for every Borel set $A$ with $|A| = 0$;
\item\label{systempde} The following system holds in the sense of distributions:
\[
\begin{cases}
\partial_t\Phi_t^b(z) = b\circ \Phi_t^b(z)\\
\Phi_0(z) = z
\end{cases}
\]
\end{enumerate}
Notice that the field we are considering is autonomous. The theory is developed also for non-autonomous vector-fields, i.e. in general $b = b(t,z)$, but we will only need to consider the autonomous case in this paper. One of the main achievements of the DiPerna-Lions-Ambrosio theory is showing that for every vector field $b \in BV\cap L^\infty(\R^n,\R^n)$ with bounded divergence, such a flow exists and is unique, see \cite[Theorem 6.2-6.4]{AMI}. Furthermore, Regular Lagrangian Flows are stable, in the sense that if $b_k$ is a sequence of smooth fields converging strongly in $L^1_{\loc}$ to $b$, and $\{\dv b_k\}_{k \in \N}$ is equibounded in $L^\infty$, then $\Phi_t^{b_k}$ converges strongly in $L^1_{\loc}$ to $\Phi_t^b$ for every $t \in \R$, see \cite[Theorem 6.6]{AMI} or \cite[Theorem 5.2]{CFL}.
\\
\\
Condition \eqref{meas} in the axioms defining the Regular Lagrangian flow guarantees the existence of a density $\xi$ for which
\begin{equation}\label{density}
\int_{\R^{n + 1}}f(t,\Phi^b_{-t}(z))dzdt = \int_{\R^{n+ 1}}f(t,z)\xi(t,z)dzdt
\end{equation}
for every $f \in C_c(\R^{n + 1})$. Clearly, the function $\xi$ depends on $b$, but we will drop the dependence, as it will always be clear from the context which density we are using. Of course, in the classical case, i.e. if $b$ were smooth, one would have
\[
\xi(t,z) = \det(D\Phi^b_t(z)).
\]
In the smooth case, Liouville Theorem tells us that
\begin{equation}\label{Liou}
\partial_t\det(D\Phi^b_t(z)) = \dv(b)\circ \Phi_t^b\det(D\Phi^b_t(z)).
\end{equation}
The latter yields, through Gronwall inequality,
\begin{equation}\label{LINF}
e^{-T\|\dv b\|_\infty}\le \det(D\Phi_t^b(z)) \le e^{T\|\dv b\|_\infty}, \forall t \in [-T,T], T \in (0, + \infty).
\end{equation}
The stability property of the Regular Lagrangian flows and \eqref{LINF} allow us to conclude that, if $b_\eps = b\star \rho_\eps$ denotes the mollification of $b$,
\[
\det(D\Phi_t^{b_\eps}(\cdot)) \overset{*}{\rightharpoonup} \xi(\cdot,t)
\]
in $L^\infty$, and that moreover
\begin{equation}\label{boundxi}
e^{-T\|\dv b\|_\infty}\le\|\xi\|_{L^\infty((-T,T)\times\R^n)}\le e^{T\|\dv b\|_\infty},
\end{equation}
so that in particular $\xi \in L^\infty([-T,T]\times \R^n)$ for every $T > 0$ and for a.e. $z \in \R^n$,
\begin{equation}\label{xi1}
\xi(0,z) = 1.
\end{equation}
A similar approximation argument also tells us that Liouville Theorem still holds, i.e. for a.e. $z \in \R^n$, $t\mapsto \xi(t,z)$ is a Lipschitz curve that satisfies for a.e. $t \in \R$
\begin{equation}\label{weakLi}
\partial_t\xi = \dv(b)\circ \Phi_t^b\xi.
\end{equation}
The latter can be equivalently intended in the sense of distributions. By \eqref{weakLi} we deduce the bound
\begin{equation}\label{boundlip}
\esssup_{z\in \R^n}\|\xi(\cdot,z)\|_{\Lip([-T,T])}\le C(T,\|\dv Y\|_\infty).
\end{equation}

Finally, we will need a Lemma that is a direct consequence of the \emph{commutator estimate} of DiPerna-Lions. This fact is probably well-known, but we provide a proof (in the appendix) since we were not able to find a reference. 

\begin{lemma}\label{lemmag}
Let $f \in L^1(\R^n)$, $a \in L^\infty(\R^n,\R^m)$, $b \in W^{1,p}(\R^n,\R^n)$ with bounded divergence, $p\ge 1$, and suppose that
\[
\dv(a\otimes b) = f
\]
in the distributional sense. Then, for every $g\in C_c^1(\R^m,\R^n)$, we have
\begin{equation}\label{g}
\dv(g(a)\otimes b) = Dg(a)f - \dv b(Dg(a)a - g(a)),
\end{equation}
in the distributional sense.
\end{lemma}

\section{A bridge made by two distributions}\label{mod}
%

We consider the following distributions of order $1$
\begin{equation}\label{distweaks}
T_{t,s}
[X,Y]
(\varphi) =\int_{\R^n}(Y(\Phi_t^X\circ\Phi_s^Y),\varphi) + (\Phi_t^X\circ\Phi_s^Y,D\varphi Y) + \dv(Y)(\varphi,\Phi_t^X\circ\Phi_s^Y)dz,
\end{equation}
and
\begin{equation}\label{distweak}
T_t[X,Y](\varphi) =\int_{\R^n}(Y\circ\Phi^X_t,\varphi) + (\Phi_t^X,D\varphi Y) + \dv(Y)(\varphi,\Phi_t^X)dz.
\end{equation}
In both expressions, $\varphi \in C^\infty_c(\R^n,\R^n)$. In the case of smooth vector fields, they represent the functions $Y(\Phi_t^X\circ\Phi_s^Y)- D(\Phi_t^X\circ\Phi_s^Y)Y(z)$ and $Y\circ\Phi^X_t(z) - D\Phi_t^X(z)Y(z)$ respectively.
\\
\\
We first claim that the commutativity of flows is equivalent to the fact that either of these distributions is identically $0$.
\begin{prop}\label{prop:1}
 The following assertions are equivalent:
\begin{enumerate}
\item $\Phi_t^X\circ\Phi_s^Y = \Phi_s^Y\circ \Phi_t^X ${ for every }$ s,t\in \R,$ a.e. in $\R^n;$
\item
$T_{t,s} \equiv 0$ for every  $t,s \in \R$;
\item
$T_{t} \equiv 0$ for every  $t \in \R$.
\end{enumerate}
Moreover it holds
\begin{equation}
\label{eqn:deriv-brack}
\left.\frac{d}{dt}\right|_{t = 0}T_t[X,Y] = [X,Y]\llcorner dz;
\end{equation}
in particular, if $T_t \equiv 0$ then $[X,Y]\equiv 0$.
\end{prop}

Next we claim that the missing implication, namely that $[X,Y]\equiv 0$ implies any among (1), (2) or (3), holds under the additional assumption on the differentiability of the flow in the direction of the other vector field

\begin{prop}\label{eqcond}
Suppose $[X,Y] = 0$ a.e. and that the flow map $\Phi_t^X$ is weakly (Lie) differentiable in direction $Y$ with derivative $f \in L^\infty_{\loc}(\R,L_{\loc}^{p'}(\R^{n},\R^n))$ for all $t\in \R$ and for $p'$ H\"older conjugate of $p$, with $p$ such that $X \in W^{1,p}_{\loc}$.
Then,
\[
T_t[X,Y] \equiv 0.
\]
\end{prop}

It is clear that Theorem~\ref{thm:main} follows from Propositions~\ref{prop:1} and~\ref{eqcond}. We devote the following sections to prove the various implications stated in these propositions.

%

\subsection{{Proof of \texorpdfstring{$\Phi_t^X\circ\Phi_s^Y = \Phi_s^Y\circ\Phi_t^X \Longleftrightarrow T_{t,s} \equiv 0$}{First}}}\label{sec:i-ii}

We claim that the fact that $\Phi_t^X\circ\Phi_s^X = \Phi_s^Y\circ\Phi_t^X$ for every $t,s \in \R$ is equivalent to say that for a.e. $z$ the function $s\to \Phi_t^X\circ\Phi_s^Y(z)$ is absolutely continuous and
\begin{equation}
\label{eqn:equiv-ode}
\partial_s(\Phi_t^X\circ\Phi_s^Y) = Y(\Phi_t^X\circ \Phi_s^Y) \qquad \mbox{for a.e. }s\in \R.
\end{equation}
Indeed, let us assume that $\Phi_t^X\circ\Phi_s^X = \Phi_s^Y\circ\Phi_t^X$ for every $t,s$. Then for every $t$ and a.e. $z$ the function $s\to \Phi_t^X\circ\Phi_s^Y(z)$ is absolutely continuous and we can write
\[
\partial_s(\Phi_t^X\circ\Phi_s^Y) = \partial_s(\Phi_s^Y\circ\Phi_t^X) = Y(\Phi_s^Y\circ\Phi_t^X) =  Y(\Phi_t^X\circ\Phi_s^Y) \qquad \mbox{for a.e. }s\in \R,
\]
 On the other hand, suppose that in the weak sense \eqref{eqn:equiv-ode} holds. We observe that $\Phi_t^X\circ\Phi_s^Y\circ\Phi_{-t}^X(z)$ has bounded density with respect to the Lebesgue measure and it coincides with $z$ at $s=0$ for a.e. $z$. Evaluating  \eqref{eqn:equiv-ode} at $\Phi_{-t}^X(z)$ we deduce that $\Phi_t^X\circ\Phi_s^Y\circ\Phi_{-t}^X(z)$ coincides with the unique Regular Lagrangian flow of $Y$, namely with $\Phi_s^Y(z)$.
\\
\\
We use this observation in the following way. Equation \eqref{eqn:equiv-ode} means, for any $\varphi \in C^\infty_c(\R^{n + 1})$,
\[
-\int_{\R^{n + 1}}(\partial_s \varphi(s,z),\Phi_t^X\circ\Phi_s^Y)dsdz = \int_{\R^{n + 1}}(Y(\Phi_t^X\circ \Phi_s^Y),\varphi(s,z))dsdz.
\]
Making the change of variable $a = \Phi_s^Y(z)$, we write
\begin{equation}\label{change}
-\int_{\R^{n + 1}}(\partial_s\varphi(s,\Phi_{-s}^Y(a)),\Phi_t^X(a))\xi(-s,a)dsda = \int_{\R^{n + 1}}(Y(\Phi_t^X\circ \Phi_s^Y),\varphi(s,z))dsdz.
\end{equation}
We can compute
\[
\partial_s(\varphi(s,\Phi_{-s}^Y(a))) = \partial_s\varphi(s,\Phi_{-s}^Y(a)) - D\varphi(s,\Phi_{-s}^Y(a))Y\circ \Phi^Y_{-s}(a).
\]
Furthermore, using \eqref{weakLi}, we infer
\[
\int_{\R^{n + 1}}(\partial_s(\varphi(s,\Phi_{-s}^Y(a))),\Phi_t^X(a))\xi(-s,a)dsda = \int_{\R^{n + 1}}(\varphi(s,\Phi_{-s}^Y(a)),\Phi_t^X(a))\dv(Y)\circ \Phi_{-s}^Y\xi(-s,a)dsda.
\]
Exploiting the last two equalities, \eqref{change} rewrites as
\begin{align*}
-\int_{\R^{n + 1}}&(D\varphi(s,\Phi_{-s}^Y(a))Y\circ \Phi^Y_{-s}(a),\Phi_t^X(a))\xi(-s,a)dsda \\
&- \int_{\R^{n + 1}}(\varphi(s,\Phi_{-s}^Y(a)),\Phi_t^X(a))\dv(Y)\circ \Phi_{-s}^Y\xi(-s,a)dsda = \int_{\R^{n + 1}}(Y(\Phi_t^X\circ \Phi_s^Y),\varphi(s,z))dsdz.
\end{align*}
We can now change back to the variable $z = \Phi_{-s}^Y(a)$ in the previous equation to find
\begin{align*}
-\int_{\R^{n + 1}}&(D\varphi(s,z)Y,\Phi_t^X\circ \Phi_s^Y)dsdz \\
&- \int_{\R^{n + 1}}(\varphi(s,z),\Phi_t^X\circ \Phi_s^Y)\dv(Y)dsdz = \int_{\R^{n + 1}}(Y(\Phi_t^X\circ \Phi_s^Y),\varphi(s,z))dsdz.
\end{align*}
Now considering functions of the form $\varphi(s,z) = \alpha(s)\beta(z)$, the previous equality reads as
\begin{equation}\label{ann}
\int_\R\alpha(s)\int_{\R^n}\left[(D\beta(z)Y,\Phi_t^X\circ \Phi_s^Y) + (Y(\Phi_t^X\circ \Phi_s^Y),\beta(z)) + \dv(Y)(\Phi_t^X\circ \Phi_s^Y,\beta(z))\right]dzds = 0.
\end{equation}
The latter is equivalent to say that for every $\beta \in C^\infty_c(\R^n,\R^n)$, $t \in \R$ and a.e. $s \in \R$,
\begin{equation}\label{betazer}
\int_{\R^n}(D\beta(z)Y,\Phi_t^X\circ \Phi_s^Y) + (Y(\Phi_t^X\circ \Phi_s^Y),\beta(z))dz =- \int_{\R^n}\dv(Y)(\Phi_t^X\circ \Phi_s^Y,\beta(z))dz,
\end{equation}
Notice that the set of $s$ on which the latter may not hold a priori depends on $\beta$, but this dependence can be eliminated by taking a dense countable subset of smooth test functions. In order to see that \eqref{betazer} is valid for every $s \in \R$, and not only up to a set of null measure, 
we observe that the left hand side is absolutely continuous  in $s$, as can be seen changing again variable $a = \Phi_s^Y(z)$ and using the dominated convergence theorem. Moreover, also the right-hand side is continuous in $s$, since for every $\eps>0$ it can be rewritten as  
$$ \int_{\R^n}\dv(Y)(\Phi_t^X\circ \Phi_s^Y,\beta(z))dz =
 \int_{\R^n}\dv(Y)((\Phi_t^X \ast \rho_\eps )\circ \Phi_s^Y,\beta(z))dz +  \int_{\R^n}\dv(Y)(( \Phi_t^X - \Phi_t^X \ast \rho_\eps )\circ \Phi_s^Y,\beta(z))dz 
 $$
 and the first term is continuous in $s$, while the second goes to $0$ as $\eps \to 0$ (locally) uniformly in $s$.
Hence \eqref{betazer} is valid for every $s$. We conclude by saying that $\eqref{betazer}$ is equivalent to
$T_{t,s}\equiv 0.$

\subsection{Proof of \texorpdfstring{$T_{t}\equiv 0 \Longleftrightarrow T_{t,s}\equiv 0$}{Second}}
We first observe that, if  $T_{t,s}\equiv 0$ and we evaluate it at $s=0$, we find that $T_t \equiv 0$. This observation, together with the equivalence proved in Section~\ref{sec:i-ii} and applied with $X=Y$, implies in particular that \eqref{YY} holds.
\\
\\
To prove the opposite implication, we start with the fact that $T_s[Y,Y] = 0$, that is formula \eqref{YY}, namely
\[
\dv(\Phi_s^Y\otimes Y) = Y\circ\Phi_s^Y + \dv Y\Phi_s^Y
\]
in the sense of distributions.
Now, Lemma \ref{lemmag} tells us that for any $g \in C^1_c(\R^n)$,
\begin{equation}\label{renorm}
\dv(g(\Phi_s^Y)\otimes Y) = Dg\circ\Phi_s^Y(Y\circ\Phi_s^Y + \dv Y\Phi_s^Y) - \dv Y(Dg(\Phi_s^Y)\Phi_s^Y - g(\Phi_s^Y))  = Dg\circ\Phi_s^YY\circ\Phi_s^Y + \dv Y g(\Phi_s^Y),
\end{equation}
that has again to be intended in distributional sense. If the flow $\Phi_t^X$ was smooth in space, we would consider $g = \Phi_t^X$; the above would become exactly $T_{t,s} \equiv 0$, once we exploit $T_t \equiv 0$ to write $D\Phi_t^X(\Phi_s^Y)Y(\Phi_s^Y) = Y(\Phi_t^X\circ\Phi_s^Y)$. This choice for $g$ is however not allowed since in the expression \eqref{renorm} the gradient of $g$ appears and flows of Sobolev vector fields need not have any weak derivative. Therefore, we fix $\rho \in C^\infty_c(B_1)$ and let $\rho_\eps$ be the associated convolution kernels. We choose as $g$ a mollification of the flow
\[
g(z) = \Phi_t^X\star \rho_\eps,
\]
to rewrite \eqref{renorm} as
\begin{equation}
\label{eqn:eps}
\dv((\Phi_t^X\star \rho_\eps)\circ\Phi_s^Y\otimes Y) = D(\Phi_t^X\star \rho_\eps)\circ\Phi_s^YY\circ\Phi_s^Y + \dv Y \,(\Phi_t^X\star \rho_\eps)\circ\Phi_s^Y.
\end{equation}
We want to take the limit as $\eps \to 0$ in the previous equality. It is immediate to see, by writing explicitely the weak form of the above equation, that in the sense of distributions
\[
\dv((\Phi_t^X\star \rho_\eps)\circ\Phi_s^Y\otimes Y) \to \dv(\Phi_t^X\circ\Phi_s^Y\otimes Y),\quad \text{ for }\eps \to 0,
\]
and strongly in $L^1$
\[
\dv Y \,(\Phi_t^X\star \rho_\eps)\circ\Phi_s^Y \to \dv Y \,\Phi_t^X\circ\Phi_s^Y,\quad \text{ for }\eps \to 0.
\]
Therefore, we need to prove that weakly in $L^1$
\begin{equation}\label{weakcon}
D(\Phi_t^X\star \rho_\eps)Y \rightharpoonup Y(\Phi_t^X) \qquad \mbox{ weakly in $L^1$ as } \eps \to 0.
\end{equation}
This in turn implies that
\[
D(\Phi_t^X\star \rho_\eps)\circ\Phi_s^YY\circ\Phi_s^Y \rightharpoonup Y(\Phi_t^X\circ \Phi_s^Y)
\]
and hence we can take the limit in the distributional formulation of \eqref{eqn:eps} to get that $T_{t,s} \equiv 0$.
\\
\\
Take any $f \in L^{p'}(\R^n,\R^n)$, where $\frac{1}{p'}+\frac{1}{p} = 1$. We can write
\begin{align*}
\int_{\R^n}(f(z),D(\Phi_{t}^X&\star\rho_\eps)(z)Y(z))dz = \int_{\R^n}\int_{\R^n}(f(z),\Phi_{t}^X(z - y))(D\rho_\eps(y),Y(z))dzdy\\
&= \int_{\R^n}\int_{\R^n}(f(z),\Phi_{t}^X(z - y))(D\rho_\eps(y),Y(z - y))dzdy \\
&+ \int_{\R^n}\int_{\R^n}(f(z),\Phi_{t}^X(z - y))(D\rho_\eps(y),Y(z)-Y(z - y))dzdy.
\end{align*}
Now we prove separately that
\begin{equation}\label{primo}
\lim_{\eps\to 0} \int_{\R^n}\int_{\R^n}(f(z),\Phi_{t}^X(z - y))(D\rho_\eps(y),Y(z - y))dzdy =  \int_{\R^n}(f(z),Y\circ \Phi_{t}^X)dz + \int_{\R^n}\dv Y(f,\Phi_t^X)dz
\end{equation}
and
\begin{equation}\label{secondo}
 \lim_{\eps\to 0}\int_{\R^n}\int_{\R^n}(f(z),\Phi_{t}^X(z - y))(D\rho_\eps(y),(Y(z)-Y(z - y)))dzdy =  -\int_{\R^n}\dv Y(f,\Phi_t^X)dz,
\end{equation}
and this concludes the proof. To show \eqref{primo}, we exploit $T_t \equiv 0$ to write
\begin{align*}
 \int_{\R^n}\int_{\R^n}(f(z),\Phi_{t}^X(z - y))(D\rho_\eps(y),Y(z - y))dzdy 
&= \int_{\R^n}\int_{\R^n}(f(z),\Phi_{t}^X(y))(D\rho_\eps(z-y),Y(y))dydz\\
&=\int_{\R^n}(D(f\star\rho_\eps)(y)Y(y),\Phi_{t}^X(y))dy\\
& \overset{T_t\equiv 0}{=} \int_{\R^n}(Y\circ\Phi_{t}^X,f\star\rho_\eps)dy + \int_{\R^n}\dv Y(\Phi_{t}^X,f\star\rho_\eps)dy,
\end{align*}
and this proves \eqref{primo}. Next, to show \eqref{secondo}, we introduce the function
\[
F(z,y) \doteq  \frac{Y(z)-Y(z - y) - (DY(z),y)}{|y|},
\]
in order to write
\begin{equation}\label{lastsplit}
\begin{split}
 \int_{\R^n}\int_{\R^n}(f(z),\Phi_{t}^X(z - y))(D\rho_\eps(y),Y(z)-Y(z - y))dzdy &=  \int_{\R^n}\int_{\R^n}(f(z),\Phi_{t}^X(z - y))(|y|D\rho_\eps(y),F(z,y))dzdy\\
&+ \int_{\R^n}\int_{\R^n}(f(z),\Phi_{t}^X(z - y))(D\rho_\eps(y),DY(z)y)dzdy.
\end{split}
\end{equation}
The first addendum is bounded by
\begin{align*}
\int_{\R^n}\int_{\R^n}|f(z)||\Phi_{t}^X(z - y))||y||D\rho_\eps(y)||F(z,y)|dzdy &\le
 \|\Phi_{t}^X\|_\infty\|f\|_{p'}\int_{B_\eps}\|y\|\|D\rho_\eps(y)\|\|F(\cdot,y)\|_{p}dy.
\end{align*}
Now, since $Y \in W^{1,p}$, for every $\alpha >0$, we can find $\eps>0$ such that
\[
\|F(\cdot,y)\|_{p} \le \alpha
\]
provided $\|y\|\le \eps$. Since
\begin{equation}\label{equiboundL}
\int_{\R^n}\|y\|\|D\rho_\eps(y)\|dy \le L
\end{equation}
for $L > 0$ depending only on $\rho$ (but otherwise independent of $\eps$), this show that the first addendum converges to $0$ as $\eps \to 0$. To estimate the second, we first observe that
\[
\int_{\R^n}(D\rho_\eps(y),DY(z)y)dy = \sum_{ij}\int_{\R^n}\partial_i\rho_\eps(y)\partial_{i}Y^jy_jdy = -\sum_{ij}\int_{\R^n}\rho_\eps(y)\partial_{i}Y^j\delta_{ij}dy = -\dv Y(z)
\]
if $\delta_{ij}$ denotes Kronecker's delta. This observation lets us rewrite the second addendum of \eqref{lastsplit} as
\begin{align*}
 \int_{\R^n}\int_{\R^n}(f(z),&\Phi_{t}^X(z - y))(D\rho_\eps(y),DY(z)y)dzdy =  \\
&\int_{\R^n}\int_{\R^n}(f(z),\Phi_{t}^X(z - y) - \Phi_{t}^X(z))(D\rho_\eps(y),DY(z)y)dzdy - \int_{\R^n}\dv Y(f,\Phi_{t}^X)dz
\end{align*}
and, similarly as above, we estimate the first term as
\begin{align}
\label{eqn:conver}
\int_{\R^n}\int_{\R^n}|f(z)||\Phi_{t}^X(z - y) &- \Phi_{t}^X(z)||D\rho_\eps(y)||y||DY(z)|dzdy= \int_{\R^n}g_\eps(z)|f(z)||DY(z)|dz,
\end{align}
where
\[
g_\eps(z) \doteq \int_{\R^n}|\Phi_{t}^X(z - y) - \Phi_{t}^X(z)||D\rho_\eps(y)||y|dy.
\]
We see that
\begin{align*}
\int_{\R^n}g_\eps(z)dz &= \int_{\R^n}\int_{\R^n}|\Phi_{t}^X(z - y) - \Phi_{t}^X(z)||D\rho_\eps(y)||y|dydz \\
&= \int_{\R^n}|D\rho_\eps(y)||y|\left(\int_{\R^n}|\Phi_{t}^X(z - y) - \Phi_{t}^X(z)|dz\right)dy\\
&= \int_{B_\eps(0)}|D\rho_\eps(y)||y|\left(\int_{\R^n}|\Phi_{t}^X(z - y) - \Phi_{t}^X(z)|dz\right)dy.
\end{align*}
Using the continuity in the $L^1$ topology of the translation operation $y \mapsto \Phi_t^X(\cdot-y)-\Phi_t^X(\cdot)$, and \eqref{equiboundL}, we easily see that $\|g_\eps\|_{L^1}\to 0$ as $\eps \to 0$. To show that the quantity in \eqref{eqn:conver} converges to $0$,
it suffices to show that for any subsequence $g_{\eps_n}$ that convergence pointwise a.e. to $0$, one has
\[
\lim_{n\to \infty}\int_{\R^n}g_{\eps_n}(z)|f(z)||DY(z)|dz \to 0.
\]
The latter is an easy consequence of the dominated convergence theorem, that we can apply since by {\eqref{equiboundL}} for a.e. $z \in \R^n$ we have
$|g_\eps(z)|{\le} 2L\|\Phi_t^X\|_\infty$. This concludes the proof.


\subsection{Proof of \texorpdfstring{\eqref{eqn:deriv-brack}}{Third} and conclusion}

For any test function $\varphi$, we can differentiate in $t$ inside the integral sign in \eqref{distweak} to get
\[
\frac{d}{dt}T_t = \int_{\R^n}(DY(\Phi^X_t(z))X\circ\Phi_t^X,\varphi(z)) + (X\circ\Phi_t^X,D\varphi Y) + \dv(Y)(\varphi,X\circ\Phi_t^X)dz.
\]
Evaluating the previous expression at $t = 0$, we get
\begin{equation}\label{derzero}
\left.\frac{d}{dt}\right|_{t = 0}T_t[X,Y](\varphi) = \int_{\R^n}(DYX,\varphi) + (X,D\varphi Y) + \dv(Y)(\varphi,X)dz.
\end{equation}
Integrating the expression
\begin{align*}
(X(z),D\varphi Y)  &= \sum_{i,j}X^i\partial_j\varphi^iY^j 
= \sum_{i,j}\partial_j(X^i\varphi^iY^j) - \sum_{i,j}\partial_jX^i\varphi^iY^j -\sum_{i,j}X^i\varphi^i\partial_jY^j,
\end{align*}
 we obtain that 
\[
\int_{\R^{n}}(X,D\varphi Y)dz =-\int_{\R^n} (DXY,\varphi)- \int_{\R^n}(X,\varphi)\dv(Y)dz.
\]
Substituting the latter into \eqref{derzero}, we finish the proof.


\subsection{Proof of Corollary~\ref{chainrule}} 
let us study the directional differentiability of $\Phi_t^X$. First, we need the following technical lemma. 
\begin{lemma}\label{ra}
Let $f$ be approximately differentiable a.e., with approximate differential $df(z)$, let $p \ge 1$, $Y \in W^{1,p}_{\rm loc}\cap L^\infty(\R^n,\R^n)$ be a vector field with bounded divergence and let $\Phi_t^X,\Phi_s^Y$ be its Regular Lagrangian Flow. Then, the family of functions
\[
F_h(z) \doteq \frac{|f(\Phi^Y_h(z))-f(z)-h(df(z),Y(z))|}{|h|}
\]
converges in measure to $0$ on compact sets as $h \to 0$, i.e. for every $\eps > 0$ and for every $R > 0$, we have
\[
|\{z\in B_R(0):F_h(z) > \eps\}| \to 0, \text{ as }h\to 0.
\]
\end{lemma}

\begin{proof}
Let us fix $\eps,R > 0$. By the definition of approximate differentiabily, we find a sequence of Borel sets $A_n$ and a sequence of $C^1$ functions $f_n:\R^n\to \R$ such that
\begin{equation}
\label{eqn:def-fn}
|B_R\setminus A_n| \le \frac{1}{n} \quad\text{ and } \quad f|_{A_n}=f_n.
\end{equation}
We wish to show that, fixed $\gamma > 0$, there exists a $\delta > 0$ such that if $|h|\le \delta$, then
\[
|\{z \in B_R(0): F_h(z) > \eps\}|\le \gamma.
\]
To do so, we split the set according to the fact that $z$ and $\Phi^Y_h(z) $ belong to $A_n$ or not:
\begin{align*}
\{z \in B_R: F_h(z) > \eps\} 
&\subseteq (B_R\setminus A_n)\cup\Phi^Y_{-h}( B_R\setminus A_n)
\\
&\quad\quad\quad\quad\quad\quad\quad\quad\quad 
\cup \{z \in A_n: F_h(z) > \eps \text{ and } \Phi^Y_h(z) \in A_n\}\\
\end{align*}
The measure of the first two sets is small: indeed, the flow is measure preserving and by \eqref{eqn:def-fn}, we have that for every $n$ sufficiently large
\begin{align}\label{eqn:n-choice}
| B_R\setminus A_n| + |\Phi^Y_{-h(B_R\setminus A_n)}| \le C(Y) |B_R\setminus A_n| \le \frac{C(Y)}{n} \leq \frac{\gamma}{2}
\end{align}
We can therefore choose $n \in \N$ in such a way that \eqref{eqn:n-choice} holds
and from now on we consider it fixed. Finally, we need to estimate $|E|$, where
\[
E\doteq \{z \in A_n: F_h(z) > \eps \text{ and } \Phi^Y_h(z) \in A_n\}. 
\]
Since $f_n$ is $C^1(\R^n)$, and the estimate
\[
\|\Phi^Y_\tau(z) - z\|\lesssim {\|Y\|_\infty}\tau
\]
holds, we obtain that the $L^\infty$ norm of the flow is bounded uniformly for $\tau \leq 1$ and 
we infer the existence of $\delta_1 > 0$ such that
\begin{equation}
\sup_{\tau:|\tau|\in (0,\delta_1)}\sup_{z \in B_{2R}(0)}\|Df_n(\Phi^Y_\tau(z)) - Df_n(z)\|\le \frac{\eps}{2\|Y\|_{L^\infty(B_R)}}.
\end{equation}
Moreover, it is a classical fact that for a.e. $z \in A_n$, we have $df(z) = Df_n(z)$. Hence, if we consider any $z$ for which $df(z) = Df_n(z)$ and any $h$ for which $|h| \le \delta_1$, we can write
\begin{equation}\label{F}
\begin{split}
F_h(z) &= \frac{\|f(\Phi_h^Y(z))-f(z)-h(df(z),Y(z))\|}{|h|} =\frac{\|f_n(\Phi_h^Y(z))-f_n(z)-h(Df_n(z),Y(z))\|}{|h|}\\
& \le \fint_0^h \|Df_n(\Phi_\tau^Y(z))Y(\Phi_\tau^Y(z))-Df_n(z)Y(z)\|d\tau \\
&\le\fint_0^h \|Y\|_{L^\infty(\Phi_\tau^Y(B_R))}\|Df_n(\Phi_\tau^Y(z)) -Df_n(z)\| + \|Df_n\|_{L^\infty(B_R)}\fint_0^h\|Y(\Phi_\tau^Y(z))-Y(z)\|d\tau\\
&\le\|Y\|_{L^\infty(B_{(1 +C)R})}\fint_0^h \|Df_n(\Phi_\tau^Y(z)) -Df_n(z)\| + \|Df_n\|_{L^\infty(B_R)}\fint_0^h\|Y(\Phi_\tau^Y(z))-Y(z)\|d\tau\\
&\le \frac{\eps}{2} + \|Df_n\|_{L^\infty(B_R)}\fint_0^h\|Y(\Phi_\tau^Y(z))-Y(z)\|d\tau.
\end{split}
\end{equation}

We denote by $Y_\eps \doteq Y\star \rho_\eps$ the  mollification of $Y$. We have, using the Lipschitz regularity of the flow with respect to the time variable,
\begin{align*}
\int_{B_R}\fint_0^h\|Y_\eps(\Phi^Y_\tau(z))-Y_\eps(z)\|d\tau dz &= \int_{B_R}\fint_0^h\left\|\int_{0}^\tau DY_\eps(\Phi^Y_s(z))Y_\eps(\Phi^Y_s(z))ds\right\|d\tau dz\\
&\le \fint_0^hd\tau\int_{0}^\tau ds\int_{B_R(0)}\|DY_\eps(\Phi^Y_s(z))Y_\eps(\Phi^Y_s(z))\|dz.
\end{align*}
Now we perform the change of variable $\Phi_s^Y(z) = a$ to estimate
\begin{align*}
\int_{B_R}\fint_0^h\|Y_\eps(\Phi^Y_\tau(z))-Y_\eps(z)\|d\tau dz &\lesssim_{\|\xi\|_\infty} \fint_0^hd\tau\int_{0}^\tau ds\int_{\Phi_s^Y(B_R)}\|DY_\eps(a)Y_\eps(a)\|dz\\
&\le h\|Y_\eps\|_{L^\infty(B_{(1+C)R})}\|DY_\eps\|_{L^1(B_{(1+C)R})} \le h\|Y\|_{L^\infty(B_{(1+C)R})}\|DY\|_{L^1(B_{(1+C)R})}.
\end{align*}
Recall that the $\xi$ appearing in the previous chain of inequalities is the density of the Regular Lagrangian flow introduced in \eqref{density}. 
Now, finally, letting $\eps \to 0$, by Fatou's Lemma we find
\[
\int_{B_R}\fint_0^h\|Y(\Phi^Y_\tau(z))-Y(z)\|d\tau dz \le h\|Y\|_{L^\infty}\|DY\|_{L^1}.
\]
Now Chebyschev inequality implies that for every $\alpha>0$ 
\begin{equation}\label{claim}
E'\doteq \left\{z\in B_R: \fint_0^h\|Y(\Phi_\tau^Y(z))-Y(z)\|d\tau > \alpha\right\}\lesssim_{\|Y\|_{W^{1,p}\cap L^\infty}}\frac{h}{\alpha}.
\end{equation}


\noindent We choose
\[
\alpha \doteq \frac{\eps}{2\|Df_n\|_{L^\infty(B_R)}},
\]
so that by \eqref{F} we have $E\subseteq E'$, and we choose $\delta_2 > 0$ such that if $|h| \le \delta_2$,
\[
|E| \leq |E'|\le \frac{\gamma}{2}.
\]
This concludes the proof.
\end{proof}

We now prove Corollary~\ref{chainrule}. {By the group property of the flow and \eqref{boundxi}, it is sufficient to show that $\Delta_h(\cdot,0,\cdot)$ weakly converges to $Y\circ \Phi_{\cdot}^X(\cdot)$.
To show this, we first prove that for any $q, R,T$, $\|\Delta_h(\cdot,0,\cdot)\|_{L^q([-T,T]\times B_R)}$ is equibounded independently of $h$. We assume, without loss of generality, $q > 1$. We use the hypothesis that the flows commute and the fact that all curves $h\mapsto \Phi_h^Y(a)$ solve 
$ \partial_h\Phi_h^Y(a) = Y(\Phi_h^Y(a))$ and are $\|Y\|_\infty$-Lipschitz to estimate
\begin{align*}
\int_{-T}^T\int_{B_R}\left|\Delta_h(t,0,z)\right|^qdzdt &= \int_{-T}^T\int_{B_R}\left|\frac{\Phi_t^X(\Phi_{h}^Y(z))-\Phi_t^X(z)}{h}\right|^qdzdt \\
&= \int_{-T}^T\int_{B_R}\left|\frac{\Phi_h^Y(\Phi_{t}^X(z))-\Phi_t^X(z)}{h}\right|^qdzdt\le \|Y\|^q_\infty 2T|B_R|.
\end{align*}
We thus infer the existence of a sequence $h_j \to 0$, as $j \to \infty$, and  of a function $g \in L^\infty(\R\times \R^n)$ such that $\Delta_{h_j}(\cdot,0,\cdot)\overset{L^q}{\rightharpoonup}f(\cdot,\cdot)$, locally as above as $j \to \infty$. Now testing against any $C^1_c(\R^{n + 1},\R^n)$ vector-field $\varphi$ and again using the commutativity of the flows, we see that for each $h_j$,
\[
\int_{\R^{n + 1}}(\Delta_{h_j}(t,0,z),\varphi(t,z))dzdt = \int_{\R^{n + 1}}\left(\frac{\Phi_{h_j}^Y(\Phi_{t}^X(z))-\Phi_t^X(z)}{{h_j}},\varphi(t,z)\right)dzdt.
\]
Passing to the limit and exploiting the arbitrarity of $\varphi$, we deduce $g = Y\circ \Phi_t$. In particular, the limit does not depend on the particular subsequence, and we infer that $\Delta_h(\cdot,0,\cdot) \overset{L^q}{\rightharpoonup} Y\circ \Phi_t$ as $h \to 0$. From Proposition~\ref{prop:1}(3) it then follows that the weak derivative $\Phi_t^X$ in direction $Y$ coincides with $Y(\Phi_t^X(z))$. In the case $p > 1$, we exploit \cite[Corollary 2.5]{DLC} to say that $z\mapsto \Phi_t^X(z)$ is approximately differentiable a.e. for every $t\in\R$, combined with Lemma \ref{ra}, to infer that for every $\alpha > 0$ and $R > 0$,
\[
|\{z\in B_R: |\Delta_h(t,0,z) -d\Phi_t^X(z)Y(z)|>\alpha\}|\to 0
\]
as $h \to 0$. This fact and boundedness in $L^q_{\loc}$ proved above is enough to conclude the $L^q_{\loc}$ strong convergence of $\Delta_h(t,0,z)$ to $d\Phi_t^X(z)Y(z)$. By uniqueness of the weak limit, we infer the identity $d\Phi_t^X(z)Y(z) = Y\circ \Phi_t^X(z)$.}
\subsection{
Proof of Proposition~\ref{eqcond}}
It suffices to show that 
$$f(t,z) = Y(\Phi_t^X).$$
We are going to prove that, for a.e. $z$, both $t \mapsto f(t,z) $ and $t \mapsto Y(\Phi_t^X)$ are absolutely continuous curves which solve the same ODE, linear in the unknown $A$,
\begin{equation}\label{pdemixed}
\partial_t A =DX(\Phi_t^X) A.
\end{equation}

We first observe that the curve $t\to Y(\Phi_t^X(z))$ is absolutely continuous for a.e. $z$ and its derivative is given by $DY\circ \Phi_t^XX\circ\Phi_t^X $, as can be seen by mollifying the vector fields $Y$ and $X$ and then passing to the limit by means of the stability of the Regular Lagrangian flow.
We use our assumption
$[X,Y](\Phi_t^X) = 0$ for every  $t
$ to infer that
\begin{equation}\label{commt}
DX(\Phi_t^X)Y(\Phi_t^X) = DY(\Phi_t^X)X(\Phi_t^X).
\end{equation}
Hence,  the curve $t \to Y(\Phi_t^X) $ is absolutely continuous and satisfies
\begin{align*}
\partial_t [Y(\Phi_t^X)] = DY\circ \Phi_t^XX\circ\Phi_t^X \overset{\eqref{commt}}{=}DX(\Phi_t^X)Y(\Phi_t^X).
\end{align*}
 To show that $f$ solves \eqref{pdemixed}, we can use directly DiPerna-Lions theory in the following way. We read \eqref{dist0} as
\[
\dv(\Phi_t^X\otimes Y) = f + \dv(Y) \Phi_t^X,
\]
for a.e. $t$. We can exploit once again Lemma \ref{lemmag}, to write, for every $g \in C^1_c(\R^n,\R^n)$
\begin{equation}\label{diplions}
\dv(g(\Phi_t^X)\otimes Y) = Dg(\Phi_t^X)(f + \dv(Y)\Phi_t^X) - \dv(Y)(Dg(\Phi_t^X)\Phi_t^X-g(\Phi_t^X)) = Dg(\Phi_t^X)f + \dv(Y)g(\Phi_t^X),
\end{equation}
that has to be intendend in the weak sense. A simple approximation procedure allows us to take $g \in W_{\loc}^{1,p}(\R^n,\R^n)$
. We can therefore choose $g = X$, hence in the distributional sense we get
\[
\dv(X(\Phi_t^X)\otimes Y) = DX(\Phi_t^X)f + \dv(Y)X(\Phi_t^X),
\]
i.e.
\begin{equation}\label{equality}
\int_{\R^n}(X(\Phi_t^X),D\varphi Y)dz = \int_{\R^n}(DX(\Phi_t^X)f + \dv(Y)X(\Phi_t^X),\varphi)dz,
\end{equation}
for any $\varphi \in C_c^\infty(\R^{n})$ and $t \in \R$.
Now observe that
\begin{equation}\label{derlip}
\frac{d}{dt}\int_{\R^n}(\Phi_t^X,D\varphi Y)dz = \int_{\R^n}(X\circ\Phi_t^X,D\varphi Y)dz,
\end{equation} 
and this can be interpreted either weakly or strongly, as the left-hand side is Lipschitz in the time variable. Therefore, we can write for any $\varphi \in C^\infty_c(\R^n),\alpha \in C^\infty_c(\R)$:
\begin{align*}
\int_{\R}\int_{\R^n}\alpha'(t)(f(t,z),\varphi)dzdt &\overset{\eqref{dist0}}{=}-\int_\R\alpha'(t)\int_{\R^n}(\Phi_t^X,D\varphi(z) Y)- \dv(Y)(\varphi, \Phi_t^X)dzdt\\
&\overset{\eqref{derlip}}{=} \int_\R\alpha(t)\int_{\R^n}(X\circ\Phi_t^X,D\varphi(z) Y) - \dv(Y)(\varphi, X\circ\Phi_t^X)dzdt\\
&\overset{\eqref{equality}}{=} \int_{\R}\alpha(t)\int_{\R^n}(DX(\Phi_t^X)f,\varphi(z))dzdt
\end{align*}
The latter chain of equality and the fact that $\varphi$ is arbitrary implies that for a.e. $z$
\begin{equation}
\label{eqn:beta}
-\int_\R \alpha'(t)f(t,z)dt=\int_\R\alpha(t)DX(\Phi_t^X)fdt.
\end{equation}
In principle, the set of $z$ for which the previous equality holds may depend on $\alpha$, but, as in \eqref{betazer}, we solve this issue by a standard argument: we obtain first \eqref{eqn:beta} for a dense family of smooth $\alpha$ and then extend it by continuity to $W^{1,1}$ functions $\alpha$. This clearly shows that $f$ is absolutely continous with respect to the time variable and that it fulfills \eqref{pdemixed}.
\\
\\
To conclude, we observe that by \eqref{pdemixed}, for a.e. $z$ the difference $v= f-Y(\Phi_t^X)$ solves the ODE $\partial_t v =DX(\Phi_t^X) v$, which implies that $\frac{d}{dt}\|v\|^2(t) = 2(v,v') \le 2\|DX(\Phi_t^X)\|\|v\|^2$. Moreover,
a simple integration in \eqref{dist0} by parts tells us that
$v(0)=f(0,z) - Y(z)$ for a.e. $z$.
 Finally, since $t \to \|DX(\Phi_t^X)(z)\|$ is integrable for a.e. $z$, we deduce by Gronwall Lemma that $v\equiv 0$.
%
%

\section{Some consequences of our equivalence}\label{EL}

Let us show some cases in which Theorem \ref{thm:main} applies. The first immediate case is the one of $X \in W^{1,\infty}_{\loc}(\R^n,\R^n)$, as already said in the introduction. In that case, $\Phi_t^X$ is locally Lipschitz and hence $D\Phi_t^X$ exists due to Rademacher Theorem. Another case, in dimension 2, is a consequence of the theory developed here and the result that will appear in the aforementioned forthcoming paper \cite{Mar20}, that we state here in a simplified way:

\begin{theorem}[E. Marconi,\cite{Mar20}]\label{Elio}
Let $p > 2$, $X \in L^\infty\cap W^{1,p}_{\loc}(\R^2,\R^2)$ with zero divergence and $X\neq 0$ everywhere on $\R^2$. Then, for every fixed $t \in \R$, $\Phi_t^X \in W^{1,p}_{\loc}(\R^2,\R^2)$, and $D\Phi_t^X \in L^\infty_{\loc}(\R,L_{\loc}^p(\R^2,\R^2))$. 
\end{theorem} 

We can deduce as a corollary of the results of the previous sections and Theorem \ref{Elio} the following:

\begin{corollary}\label{CORO}
Let $p > 2$, $X \in L^\infty\cap W^{1,p}_{\loc}(\R^2,\R^2)$ with zero divergence and $X\neq 0$ everywhere on $\R^2$. Let also $q \ge 1$, $Y \in L^\infty\cap W_{\loc}^{1,q}(\R^n,\R^n)$ with bounded divergence. If $[X,Y] = 0$ a.e., then $\Phi_t^X\circ\Phi_s^Y = \Phi_s^Y\circ\Phi_t^X$ for every $t,s \in \R$.
\end{corollary}

In other words, under some assumptions on $X$, in dimension $2$ we find the classical equivalence \eqref{result}.

\begin{proof}[Proof of Corollary \ref{CORO}]
Apply Theorem \ref{Elio} to find that $\Phi_t^X$ has Sobolev regularity and $$D\Phi_t^X \in L^\infty_{\loc}(\R,L_{\loc}^p(\R^2,\R^2)).$$ This implies that
\[
D\Phi_t^XY \in L^\infty_{\loc}(\R,L_{\loc}^p(\R^2,\R^2)).
\]
Now define $f \doteq D\Phi_t^XY$. Since $p > 2$, $p' < p$, and hence
\[
f \in L^\infty_{\loc}(\R,L_{\loc}^{p'}(\R^2,\R^2)).
\]
Therefore, we see that the hypotheses of Proposition \ref{eqcond} are fulfilled and now Proposition \ref{prop:1} finishes the proof.
\end{proof}

\begin{remark}
Theorem \ref{Elio} actually holds under the assumption of bounded divergence of the field $X$, but it is not stated like this in \cite{Mar20}. With some technical arguments we could actually infer the result of Corollary \ref{CORO} for bounded divergence fields with the results appearing in \cite{Mar20}, but this would greatly increase the technicalities of the proof, and hence we have preferred this cleaner statement.
\end{remark}

\appendix

\section{Technical results}\label{APP}

This appendix contains the proof of Lemma \ref{lemmag}, that we restate here:

\begin{lemma}
Let $f \in L^1(\R^n)$, $a \in L^\infty(\R^n,\R^m)$, $b \in W^{1,p}(\R^n,\R^n)$ with bounded divergence, $p>1$, and suppose that
\[
\dv(a\otimes b) = f
\]
in the distributional sense. Then, for every $g\in C_c^1(\R^m,\R^n)$, we have
\begin{equation}\label{gg}
\dv(g(a)\otimes b) = Dg(a)f - \dv b(Dg(a)a - g(a)),
\end{equation}
in the distributional sense.
\end{lemma}
The proof is an easy consequence of the so-called \emph{commutator estimate} of DiPerna-Lions, that asserts that
\begin{equation}\label{rcomm}
R_\eps[u,b] \doteq \dv(u_\eps b)- \dv((ub)_\eps) = \sum_{j = 1}^n u\star \partial_j\rho_\eps b_j + \sum_{j = 1}^nu\star \rho_\eps\partial_jb_j - \sum_{j = 1}^n u b_j\partial_j\rho_\eps
\end{equation}
converges strongly to $0$ in $L^1_{\loc}$, under the assumption that $u \in L^\infty(\R^n)$, $b \in W^{1,p}(\R^n,\R^n)$, $p \ge 1$, with bounded divergence, and $u_\eps = u\star \rho_\eps$ is any mollification of $u$ through even mollification kernels $\rho_\eps$. For a proof, see for instance \cite[Lemma 2.2]{CFL}.
\begin{proof}
Let us denote with $a_\eps \doteq a\star \rho_\eps$, the mollification of $a$. Since $b \in W^{1,p}$, we can compute pointwise a.e.:
\begin{equation}\label{dveps}
\dv(g(a_\eps)\otimes b) = Dg(a_\eps)Da_\eps b + \dv(b)a_\eps.
\end{equation}
We rewrite
\[
Da_\eps b = \dv(a_\eps \otimes b) - \dv(b)a_\eps,
\]
and hence \eqref{dveps} becomes:
\[
\dv(g(a_\eps)\otimes b) = Dg(a_\eps)\dv(a_\eps \otimes b) - \dv(b)Dg(a_\eps)a_\eps + \dv(b)a_\eps.
\]
Finally, adding and subtracting $Dg(a_\eps)\dv(a\otimes b)_\eps$, we have
\[
\dv(g(a_\eps)\otimes b) = Dg(a_\eps)(\dv(a_\eps \otimes b)-\dv(a\otimes b)_\eps) + Dg(a_\eps)\dv(a\otimes b)_\eps - \dv(b)Dg(a_\eps)a_\eps + \dv(b)a_\eps.
\]
Since $\dv(a\otimes b) = f \in L^1$, it is immediate to see that $Dg(a_\eps)\dv(a\otimes b)_\eps$ strongly converges to $Dg(a)\dv(a\otimes b)$. The same strong convergence holds for $- \dv(b)Dg(a_\eps)a_\eps + \dv(b)a_\eps$ towards $- \dv (b)(Dg(a)a - g(a))$. In order to show \eqref{gg}, since $g \in C^1_c$ we only need to show the strong $L^1_{\loc}$ convergence to $0$ of the term
\[
v_\eps \doteq\dv(a_\eps \otimes b)-\dv(a\otimes b)_\eps 
.
\]
To conclude, we only need to observe that the $\ell$-th component, $\ell \in \{1,\dots, m\}$, of the vector field
$v_\eps$
is given by
$v_\eps^\ell = R_\eps [a^\ell, b]$, 
 where $R_\eps$ was introduced in \eqref{rcomm}. As said, $R_\eps \to 0$ strongly in $L_{\loc}^1$, and the proof is finished.
\end{proof}


\bigskip
\textbf{ Acknowledgements}. 
The authors have been supported by the SNF Grant 182565. They also wish to thank Nicola Gigli and Camillo De Lellis for bringing this problem to their attention, and  Elio Marconi for many interesting discussions. 

\bibliographystyle{plain}
\bibliography{Commutativity}

\end{document}